\documentclass[a4paper]{amsart}
\usepackage[left=3.18cm,right=3.18cm,top=2.54cm,bottom=2.54cm]{geometry}
\usepackage{amsthm,bm,amssymb}
\usepackage{color}
\usepackage{graphicx}
\usepackage{fancyhdr}
\usepackage{verbatim}
\usepackage{mathrsfs}
\usepackage{amsfonts,amscd,amssymb,epsfig,latexsym,amsmath,enumerate,stmaryrd}
\usepackage{mathtools}
\usepackage{verbatimbox}
\usepackage{tikz}
\usepackage{tikz-cd}
\usepackage{bbding}
\usepackage{dsfont}
\usepackage{slashed}
\usepackage{cleveref}
\usepackage{setspace}
\usepackage{url}

\usepackage{yfonts}

\usepackage{graphicx}
\usepackage[utf8]{inputenc}
\usepackage[export]{adjustbox}
\usepackage{wrapfig}

\graphicspath{ {images/} }
\usepackage{float}

\usetikzlibrary{matrix,arrows,decorations.pathmorphing}
\usepackage{enumerate}
\usepackage [english]{babel}
\usepackage [autostyle, english = american]{csquotes}

\theoremstyle{definition}
\newtheorem{theorem}{Theorem}[section]
\newtheorem{lemma}[theorem]{Lemma}
\newtheorem{proposition}[theorem]{Proposition}
\newtheorem{corollary}[theorem]{Corollary}

\newtheorem{remark}[theorem]{Remark}

\newcommand{\mubar}{{\overline{\mu}}}
\newcommand{\delbar}{{\overline{\partial}}}
\newcommand{\del}{\partial}

\DeclareMathOperator{\ad}{ad}

\title{On the algebra generated by $\mubar,\delbar,\del,\mu$}
\author{Shamuel Auyeung}
\address{Stony Brook University, Department of Mathematics, 100 Nicolls Road, 11794 Stony Brook}
\email{shamuel.auyeung@stonybrook.edu}
\author{Jin-Cheng Guu}
\address{Stony Brook University, Department of Mathematics, 100 Nicolls Road, 11794 Stony Brook}
\email{jin-cheng.guu@stonybrook.edu}
\author{Jiahao Hu}
\address{Stony Brook University, Department of Mathematics, 100 Nicolls Road, 11794 Stony Brook}
\email{jiahao.hu@stonybrook.edu}
\date{}
\begin{document}

\maketitle
\begin{abstract}
    In this note, we determine the structure of the associative algebra generated by the differential operators $\mubar,\delbar,\del,\mu$ that act on complex-valued differential forms of almost complex manifolds. This is done by showing it is the universal enveloping algebra of the graded Lie algebra generated by these operators and determining the structure of the corresponding graded Lie algebra. We then determine the cohomology of this graded Lie algebra with respect to its canonical inner differential $[d,-]$, as well as its cohomology with respect to all its inner differentials.
\end{abstract}

\section{Introduction}

The exterior differential $d$ on the bigraded de Rham algebra $\Omega_M^{\bullet,\bullet}$ of complex-valued differential forms on a complex manifold $M$ splits into anti-holomorphic and holomorphic parts $\delbar$ and $\del$ of bidegrees $(0,1)$ and $(1,0)$ respectively. Thus $\Omega_M^{\bullet,\bullet}$ can be viewed as a bigraded representation of the bigraded associative algebra
\[
A_{hol}=\frac{\text{Free algebra generated by }\delbar,\del}{(\delbar^2=\del^2=\delbar\del+\del\delbar=0)}
\]
which is the exterior algebra generated by $\delbar,\del$.

In \cite{Stelzig}, the bigraded representations of $A_{hol}$ are classified. This knowledge is then applied to study the geometry and topology of complex manifolds. By decomposing $\Omega_M^{\bullet,\bullet}$ into a direct sum of indecomposable representations, one obtains straightforward combinatorial descriptions of Dolbeault, Bott-Chern, and Aeppli cohomology groups. This decomposition is also useful for predicting the degeneracy of the Fr\"olicher spectral sequence \cite{PSU20} and for building more structured (rational) homotopy theoretical models for $M$ \cite{Stelzig2}.


Here we would like to carry out a similar analysis to develop tools for studying almost complex manifolds and compare the results to the complex situation. This is driven by the following long standing question: if a closed manifold admits an almost complex structure, does it necessarily admit a complex structure? In (real) dimension 2, the answer is that this always occurs. However, in dimension 4, there are examples of manifolds admitting almost complex structures that do not admit any complex structure. The question remains open in dimensions 6 and higher and in particular, it is not known whether $S^6$ admits a complex structure though it inherits an almost complex structure from the octonions.


A major difference between the complex case and the general almost complex case is that in the almost complex case, besides $\delbar$ and $\del$, the exterior differential $d$ has two extra components $\mubar$ and $\mu$ of bidegrees $(-1,2)$ and $(2,-1)$ respectively. Hence the following questions arise naturally:
\begin{enumerate}[(i)]
    \item What is the structure of the algebra $A$ generated by the operators $\mubar,\delbar,\del,\mu$?
    \item What are all indecomposable representations of $A$?
    \item What can we say about almost complex manifolds from understanding $A$ and its representations?
    \item How do we compare the almost complex case to the complex case?
\end{enumerate}

The aim of this note is to answer the first question and partially answer the last by comparing $A$ to $A_{hol}$. Our first observation is that $A$ is the universal enveloping algebra of a graded Lie algebra $\mathfrak{g}$ (\Cref{thm0}). Similarly $A_{hol}$ is the universal enveloping algebra of a graded Lie algebra $\mathfrak{g}_{hol}$, which is the abelian Lie algebra on two generators of degree one. So the questions concerning $A$ and $A_{hol}$ can all be reduced to their corresponding graded Lie algebras. 

It turns out $\mathfrak{g}$ is infinite dimensional in contrast to the finite dimensionality of $\mathfrak{g}_{hol}$. In fact, the Lie subalgebra of $\mathfrak{g}$ generated by $\delbar$ and $\del$ is free (\Cref{thm1}). This significantly increases the difficulty in finding all indecomposable representations of $A=U\mathfrak{g}$, the universal enveloping algebra of $\mathfrak{g}$. Thus the second question (and therefore the third question as well) appears to be extremely difficult to the authors at this point. We will comment in our concluding remarks (\Cref{sec4}) on how complicated the second question can be.

On the other hand, despite their enormous difference in dimensions, $\mathfrak{g}$ and $\mathfrak{g}_{hol}$ are equivalent in a weak sense. More precisely, the natural quotient map $\mathfrak{g}\to \mathfrak{g}_{hol}$ by modding out $\mubar$ and $\mu$ is a quasi-isomorphism (\Cref{quasiiso}), where $\mathfrak{g}$ and $\mathfrak{g}_{hol}$ are equipped with the natural differential $[d,-]$, i.e. the adjoint action by $d$. In fact, there is a $2$-dimensional family of differentials on $\mathfrak{g}$ and $\mathfrak{g}_{hol}$, parametrized by a Zariski open subset of the affine cone over the twisted cubic, so that the quotient map $\mathfrak{g}\to \mathfrak{g}_{hol}$ is a quasi-isomorphism. This will be discussed in \Cref{sec3} after the structure of $\mathfrak{g}$ is determined in \Cref{sec2}.

\subsection*{Acknowlegement}
This work is stimulated by conversations with Jonas Stelzig, to whom we owe our gratitude. We thank Aleksandar Milivojevi\'c for helpful discussions. All three authors were supported by the Simons Foundation International and their home institution of Stony Brook University in the academic year 2021-2022 and with the exception of the second author, also in 2022-2023. Moreover, the first author was also partially supported by NSF DMS-1547145 and NSF DMS-1901979.

\section{The Algebraic Structure of $\mubar,\delbar,\del,\mu$}\label{sec2}
Throughout we work over the ground field $\mathbb{C}$, though most of the results only require the ground field to be of characteristic zero. The algebra $A$ in question is the free bigraded associative algebra generated by the symbols $\mubar,\delbar,\del,\mu$ of bidegrees $(-1,2),(0,1),(1,0),(2,-1)$ respectively modulo the relations generated from the bigraded components of $d^2=(\mubar+\delbar+\del+\mu)^2=0$, which are
\begin{equation}
    \begin{aligned}
    &\mubar^2=0, & \mu^2=0, \\
    &\mubar\delbar+\delbar\mubar=0, & \mu\del+\del\mu=0,\\
    &\mubar\del+\del\mubar+\delbar^2=0, & \hspace{5mm} \mu\delbar+\delbar\mu+\del^2=0,\\
    &\mubar\mu+\mu\mubar+\delbar\del+\del\delbar=0.
\end{aligned}
\end{equation}

Even though $A$ carries a bigrading, the total grading is sufficient to us for most purposes of this note. So we regard $A$ as a graded associative algebra most of the time, and occasionally use its bigrading when necessary. We will use lower index for the total degree, for example $A_k$ is the (total) degree $k$ subspace of $A$.

To understand the structure of $A$, first of all we observe that the relations in $A$ are deduced from $d^2=0$, which is equivalent to $[d,d]=0$ since $[d,d]=2d^2$. Here $[-,-]$ is the standard graded commutator given by $[x,y]=xy-(-1)^{\deg x\cdot \deg y}yx$. So the above relations can be rewritten in terms of graded commutators as
\begin{equation}\label{relations}
    \begin{aligned}
    &[\mubar,\mubar]=0, & [\mu,\mu]=0, \\
    &[\mubar,\delbar]=0, & [\mu,\del]=0,\\
    &[\mubar,\del]+\frac{1}{2}[\delbar,\delbar]=0, & \hspace{5mm} [\mu,\delbar]+\frac{1}{2}[\del,\del]=0,\\
    &[\mubar,\mu]+[\delbar,\del]=0.
\end{aligned}
\end{equation}

This means, the operators $\mubar,\delbar,\del,\mu$ also generate a graded Lie algebra $\mathfrak{g}$, which is the quotient of the free graded Lie algebra generated by the symbols $\mubar,\delbar,\del,\mu$ modulo the relations \labelcref{relations}. And our first theorem is

\begin{theorem}\label{thm0}
The associative algebra $A$ generated by $\mubar,\delbar,\del,\mu$ is the universal enveloping algebra of the graded Lie algebra $\mathfrak{g}$ they generate.
\end{theorem}

We shall prove a more general fact that if all relations in an associative algebra arise from Lie brackets of generators then it is the universal enveloping algebra of a Lie algebra. From this general fact, the above theorem follows easily as a special case.

To set up the stage, let $T_S$ be the graded tensor algebra over a field (of arbitrary characteristic) generated by a set $S$ of homogeneous elements in positive degrees and let $I_R$ be the two-sided ideal of $T_S$ generated by a set of positive degree homogeneous elements $R\subset T_S$. Denote by $L_S$ the free graded Lie algebra generated by $S$. It is well-known that $T_S$ is the universal enveloping algebra of $L_S$, and that $L_S$ is a sub-Lie-algebra of $T_S$ where $T_S$ is equipped with the standard graded commutator.

\begin{proposition}
Suppose $R\subset L_S$. Then $T_S/I_R$ is the universal enveloping algebra of $L_S/J_R$ where $J_R$ is the Lie ideal of $L_S$ generated by $R$.
\end{proposition}
\begin{proof}
Consider the graded Lie algebra homomorphism $L_S\hookrightarrow T_S\to T_S/I_R$ (since $I_R$ is a two-sided ideal, it is also a Lie ideal). Since $R\subset L_S$ is mapped to zero in $T_S/I_R$, we have a commutative diagram of graded Lie algebras
\[\begin{tikzcd}
	L_S \ar[r]\ar[d]& T_S\ar[d] \\
	L_S/J_R \ar[r]& T_S/I_R
\end{tikzcd}\]
This in turn yields, by the universal property of universal enveloping algebra, a commutative diagram of associative algebras
\[
\begin{tikzcd}
	UL_S\ar[r,equal]\ar[d] & T_S\ar[d] \\
	U(L_S/J_R)\ar[r,"\phi"] & T_S/I_R
\end{tikzcd}
\]
On the other hand, consider the algebra map $T_S=UL_S\to U(L_S/J_R)$ which takes $R\subset T_S$ to zero. Therefore we have a map $\psi: T_S/I_R\to U(L_S/J_R)$ and the following diagram commutes:
\[
\begin{tikzcd}
	UL_S\ar[r,equal]\ar[d] & T_S\ar[d] \\
	U(L_S/J_R)\ar[r,shift left=1, "\phi"] & T_S/I_R\ar[l,shift left=1, "\psi"]
\end{tikzcd}
\]
Notice that both vertical maps are surjective: the right vertical map is surjective by construction and the left vertical map is surjective due to Poincar\'e-Birkhoff-Witt theorem. To see the latter, one can choose an ordered basis of $L_S/J_R$, lift it to $L_S$ and extend it to an ordered basis of $L_S$. Now the commutatively of the diagram and the surjectivity of the vertical maps force $\phi$ and $\psi$ to be inverses.
\end{proof}

Now that we know $A$ is the universal enveloping algebra of $\mathfrak{g}$, let us turn to study the structure of $\mathfrak{g}$.

\begin{lemma}
 The Lie subalgebra $\mathfrak{h} \subset \mathfrak{g}$ generated by $\del,\delbar$ is a Lie ideal.
\end{lemma}
\begin{proof}
Since $\mathfrak{g}$ is generated by $\mubar,\delbar,\del,\mu$ and $[\delbar,-],[\del,-]$ preserve $\mathfrak{h}$, it suffices to show $[\mubar,-],[\mu,-]$ preserve $\mathfrak{h}$. Further since $\mathfrak{h}$ is generated by $\delbar,\del$, it suffices to check that $[\mubar,\delbar]$, $[\mubar,\del]$, $[\mu,\delbar]$, $[\mu,\del]$ are contained in $\mathfrak{h}$, which obviously follows from \labelcref{relations}.
\end{proof}

\begin{corollary}\label{cor1}
$\mathfrak{g}=\mathfrak{h}$ in degrees $\ge 2$, and $\mathfrak{g}/\mathfrak{h}$ is isomorphic to the free graded abelian Lie algebra on two (equivalent classes of) generators $\mubar$ and $\mu$.
\end{corollary}

Recall that we use lower index for the total degree, for example $\mathfrak{g}_k$ and $\mathfrak{h}_k$ are the total degree $k$ subspace of $\mathfrak{g}$ and $\mathfrak{h}$ respectively.

\begin{proof}
From $$[\mubar,\mubar]=0,[\mu,\mu]=0,[\mubar,\mu]=-[\delbar,\del]\in\mathfrak{h},$$
we see $\mathfrak{g}_2=\mathfrak{h}_2 = \text{span}\{[\del,\del],[\delbar,\delbar],[\del,\delbar]\}$. Then inductively using the above lemma and that 
\begin{align*}
    \mathfrak{h}_{k+1}\subset\mathfrak{g}_{k+1}&=[\mubar,\mathfrak{g}_k]+[\delbar,\mathfrak{g}_k]+[\del,\mathfrak{g}_k]+[\mu,\mathfrak{g}_k]\\
    &=[\mubar,\mathfrak{h}_k]+[\delbar,\mathfrak{h}_k]+[\del,\mathfrak{h}_k]+[\mu,\mathfrak{h}_k]\subset\mathfrak{h}_{k+1},
\end{align*}
we have $\mathfrak{g}=\mathfrak{h}$ in degrees $\ge 2$. The second assertion follows easily.
\end{proof}
\begin{theorem}\label{thm1}
The Lie subalgebra $\mathfrak{h}$ of $\mathfrak{g}$ generated by $\delbar,\del$ is a free graded Lie algebra.
\end{theorem}
\begin{proof}
Let $\mathfrak{h}'$ be the free graded Lie algebra generated on symbols $\delbar,\del$. We define a derivation $D_\mubar$ on $\mathfrak{h}'$ by first setting
\[
    D_\mubar\delbar=0, D_\mubar\del=-\frac{1}{2}[\delbar,\delbar]
\]
and then extending it to a derivation. It is easy to see that $[D_\mubar,D_\mubar]=2 D_\mubar^2=0$ on $\delbar$ and $\del$. Since $[D_\mubar,D_\mubar]$ is a derivation on $\mathfrak{h}'$ that vanishes on the generators, $[D_\mubar,D_\mubar]$ must vanish identically on $\mathfrak{h}'$.

Define $\mathfrak{g}''$ to be the semi-direct product $\mathfrak{h}'\oplus \mathbb{C} D_\mubar$ with Lie bracket inherited from that of $\mathfrak{h}'$, the action of $D_\mubar$ on $\mathfrak{h}'$ and $[D_\mubar,D_\mubar]=0$. More precisely, for homogeneous $x,y\in \mathfrak{h}'$ and $s,t\in \mathbb{C}$ we define
\[
[x+s D_\mubar,y+t D_\mubar]:=[x,y]+s \cdot (D_\mubar y)+(-1)^{\deg x}t\cdot (D_\mubar x).
\]
This Lie bracket indeed satisfies the Jacobi identity since $D_\mubar$ is a derivation on $\mathfrak{h}'$. We can similarly define a derivation $D_\mu$ on $\mathfrak{g}''$ by
\[
D_\mu \delbar=-\frac{1}{2}[\del,\del], D_\mu \del=0, D_\mu (D_\mubar)=-[\del,\delbar].
\]
Since $D_\mu^2$ vanishes on $\delbar$ and $\del$, and
\[
D_\mu^2(D_\mubar)=-D_\mu[\del,\delbar]=-[D_\mu \del,\delbar]+[\del,D_\mu \delbar]=-\frac{1}{2}[\del,[\del,\del]]=0.
\]
We see $[D_\mu,D_\mu]$ vanished identically on $\mathfrak{g}''$. Define $\mathfrak{g}'$ to be the semi-direct product $\mathfrak{g}''\oplus \mathbb{C}D_\mu$.

Now we show $\mathfrak{g}'$ is isomorphic to $\mathfrak{g}$. First, consider the map from the free graded Lie algebra generated by the symbols $\mubar,\delbar,\del,\mu$ onto $\mathfrak{g}'$ by taking $\mubar,\delbar,\del,\mu$ to $D_\mubar,\delbar,\del,D_\mu$ respectively. This map descends to an epimorphism $\mathfrak{g}\to\mathfrak{g}'$ since, by the construction of $\mathfrak{g}'$, all the relations in \labelcref{relations} are mapped to zero in $\mathfrak{g}'$. Moreover, this map $\mathfrak{g}\to\mathfrak{g}'$ by construction takes $\mathfrak{h}$ onto $\mathfrak{h}'$. We claim the map $\mathfrak{h}\to\mathfrak{h}'$ is an isomorphism. Indeed since $\mathfrak{h}'$ is free, we have a unique graded Lie algebra homomorphism $\mathfrak{h}'\to \mathfrak{h}$ defined by sending the symbols $\delbar,\del$ in $\mathfrak{h}'$ to $\delbar,\del$ in $\mathfrak{h}$ respectively. This map is inverse to the previous homomorphism $\mathfrak{h}\to \mathfrak{h}'$ since the two compositions of these two morphisms are identities on the corresponding generators. Consequently the map $\mathfrak{g}\to\mathfrak{g}'$ yields $\mathfrak{h}\cong\mathfrak{h}'$. Combining this with the above corollary, we conclude $\mathfrak{g}\cong\mathfrak{g}'$.
\end{proof}

\begin{remark}
The Lie subalgebra of $\mathfrak{g}$ generated by $\mubar,\mu$ is isomorphic to the Heisenberg (graded) Lie algebra. That is, it is the three dimensional graded Lie algebra spanned by $\mubar,\mu,[\mubar,\mu]$ whose center is spanned by $[\mubar,\mu]$ and in which $[\mubar,\mubar]=[\mu,\mu]=0$.
\end{remark}

\begin{corollary}
$U\mathfrak{g}$ is a free (left and right) module over $U\mathfrak{h}$ with basis $1,\mubar,\mu, \mubar\mu$.
\end{corollary}
\begin{proof}
That $U\mathfrak{g}$ is a free module over $U\mathfrak{h}$ follows from the Poincar\'e-Birkhoff-Witt theorem (see \cite[pp. 288]{FHT}). To prove the second assertion, note $\mathfrak{h}$ being a $\mathfrak{g}$-ideal implies that $U\mathfrak{h}$ is an (two-sided) $U\mathfrak{g}$-ideal. The quotient $U\mathfrak{g}/U\mathfrak{h}$ is the free associative algebra generated by $\mubar,\delbar,\del,\mu$ modulo the relations $\delbar=0,\del=0,[\mubar,\mu]=0$. This quotient algebra has a basis $1,\mubar,\mu,\mubar\mu$, thus proving $1,\mubar,\mu,\mubar\mu$ is a $U\mathfrak{h}$-basis for $U\mathfrak{g}$.
\end{proof}

The above corollary gives a basis for $A=U\mathfrak{g}$ as follows. Since $U\mathfrak{h}$ is the free tensor algebra on $\delbar,\del$, it has a canonical basis given by simple tensors with factors $\delbar$ and $\del$. Then a basis for $U\mathfrak{g}$ can be obtained from multiplying the basis for $U\mathfrak{h}$ by $1,\mubar,\mu,\mubar\mu$ from either left or right. The Poincar\'e series of $A$, i.e. $\sum_k (\dim A_k) \, q^k$, is $(1+q)^2/(1-2q)$.

\section{Cohomology and Maurer-Cartan Elements}\label{sec3}

Since $\mathfrak{g}$ carries a natural differential $\ad_d=[d,-]$, we would like to determine its cohomology with respect to $\ad_d$ which amounts to finding elements that commute with $d$ (are $\ad_d$-closed) but are not themselves commutators with $d$ ($\ad_d$-exact). To begin with, consider the free graded abelian Lie algebra $\mathfrak{g}_{hol}$ generated by $\del,\delbar$, and equip it with the inner differential $[\delbar+\del,-]$ (which is the zero differential since $\mathfrak{g}_{hol}$ is abelian). This Lie algebra is closely related to complex manifolds. The natural quotient map
\[
f:\mathfrak{g}\to \mathfrak{g}_{hol}
\]
obtained by modding out $\mubar,\mu$ is a morphism of differential graded Lie algebras. An almost complex manifold is integrable if and only if the canonical action of $\mathfrak{g}$ on its complex-valued differential forms descends, via $f$, to an action of $\mathfrak{g}_{hol}$. In the previous section, we see that $\mathfrak{g}$ is quite large since it contains a free Lie subalgebra while $\mathfrak{g}_{hol}$ is much smaller. 

Recall the Lie subalgebra $\mathfrak{h}$ of $\mathfrak{g}$ generated by $\del,\delbar$ is free, which is isomorphic to the homotopy Lie algebra of $\mathbb{C}\mathbb{P}^1\vee \mathbb{C}\mathbb{P}^1$ tensored with $\mathbb{C}$. Here by homotopy Lie algebra of a (pointed) topological space $X$ we mean $\pi_*(\Omega X)=\bigoplus_{n\ge 0}\pi_n(\Omega X)$ equipped with Whitehead bracket where $\Omega X$ is the based loop space of $X$. Meanwhile, the Lie subalgebra of $\mathfrak{g}_{hol}$ generated by $\del,\delbar$ is $\mathfrak{g}_{hol}$ itself and is isomorphic to the homotopy Lie algebra of $\mathbb{C}\mathbb{P}^\infty\times \mathbb{C}\mathbb{P}^\infty$ tensored with $\mathbb{C}$. From this point of view, the quotient map $f$ corresponds to the inclusion of $\mathbb{C}\mathbb{P}^1\vee \mathbb{C}\mathbb{P}^1$ into $\mathbb{C}\mathbb{P}^\infty\times \mathbb{C}\mathbb{P}^\infty$ as the 2-skeleton. It appears there is an enormous gap between $\mathfrak{g}$ and $\mathfrak{g}_{hol}$ corresponding to higher dimensional cells of $\mathbb{C}\mathbb{P}^\infty\times \mathbb{C}\mathbb{P}^\infty$.

However, we have the following surprising theorem:

\begin{theorem} \label{quasiiso}
The natural quotient morphism $f: (\mathfrak{g},\ad_d)\to(\mathfrak{g}_{hol},\ad_{\delbar+\del}\equiv 0)$ is a quasi-isomorphism of differential graded Lie algebras.
\end{theorem}

The proof of this theorem relies on the following proposition.

\begin{proposition}
$H^k(\mathfrak{h},\ad_\mubar)=0$ for $k>2$.
\end{proposition}
\begin{proof}
First of all, we note $\ad_\mubar\delbar=0$ and
\[
\ad_\mubar[\del,\delbar]=[\ad_\mubar \del,\delbar]-[\del,\ad_\mubar \delbar]=-\frac{1}{2}[[\delbar,\delbar],\delbar]=0.
\]
We leave it to the reader to check when $k=1,2$, $H^k(\mathfrak{h},\ad_\mubar)$ is one-dimensional and spanned by (the equivalence classes of) $\delbar$ and $[\del,\delbar]$ respectively. Next we observe $H(\mathfrak{h},\ad_\mubar)$ is a Lie algebra and $H^1(\mathfrak{h})\oplus H^2(\mathfrak{h})$ forms an \textit{abelian} Lie subalgebra since $[\delbar,\delbar]=-2[\mubar,\del]$ is $\ad_\mubar$-exact, and $[\delbar,[\del,\delbar]]=0$, $[[\del, \delbar],[\del,\delbar]]=0$ by Jacobi identities.

Now consider the universal enveloping algebra $UH(\mathfrak{h},\ad_\mubar)$ of $H(\mathfrak{h},\ad_\mubar)$. It contains the universal enveloping algebra of the abelian Lie subalgebra $H^1(\mathfrak{h})\oplus H^2(\mathfrak{h})$, which is the free graded commutative algebra $\Lambda(\delbar,[\del,\delbar])$ generated by $\delbar$ and $[\del,\delbar]$. Meanwhile, since the universal enveloping algebra functor commutes with cohomology (see e.g. \cite[Appx B. Prop. 2.1]{Quillen}), we have $UH(\mathfrak{h},\ad_\mubar)=H(U\mathfrak{h},\ad_\mubar)$  where $\ad_\mubar$ on $U\mathfrak{h}$ is the extended adjoint action. So we get
\[
\Lambda(\delbar,[\del,\delbar])\subset H(U\mathfrak{h},\ad_\mubar).
\]
By Poincar\'e-Birkhoff-Witt theorem, our proposition is equivalent to $\Lambda(\delbar,[\del,\delbar])= H(U\mathfrak{h},\ad_\mubar)$. This equality clearly holds in degrees $\le 2$. We will prove this by induction on degree, but we need to make some preparations.

For simplicity of notation, denote $B=U\mathfrak{h}$, which is the free tensor algebra on $\delbar,\del$ by \Cref{thm1}. Under the isomorphism
\[
\phi: B_{k-1}\oplus B_{k-1}\cong B_k, (x,y)\mapsto \del x+\delbar y,
\]
the differential $\ad_\mubar|_{B_k}$ can be written as the matrix
\begin{equation}\label{equation-interesting}
\ad_\mubar|_{B_k}\cong
\begin{pmatrix}
 -\ad_\mubar|_{B_{k-1}} & 0 \\
-\delbar|_{B_{k-1}} & -\ad_\mubar|_{B_{k-1}}
\end{pmatrix}
\end{equation}
by using the relations \labelcref{relations}. To see this, we compute for $x,y\in B_{k-1}$
\begin{equation*}
\begin{aligned}
    [\mubar, \del x+\delbar y]&=\mubar\del x-(-1)^k\del x\mubar+\mubar \delbar y-(-1)^k\delbar y\mubar\\
    &= -\delbar^2x-\del\mubar x-(-1)^k\del x\mubar-\delbar\mubar y-(-1)^k\delbar y\mubar\\
    &= -\del\left(\mubar x-(-1)^{k-1}x\mubar\right)-\delbar\left(\delbar x+\mubar y-(-1)^{k-1}y\mubar\right)\\
    &= -\del[\mubar,x]-\delbar(\delbar x+[\mubar,y]).
\end{aligned}
\end{equation*}
In particular, by setting $x=0$, we see $\ad_\mubar$ skew commutes with $\delbar$. This means both $\pm\delbar$ are morphisms of cochain complexes $\pm\delbar:B_\bullet\to B_\bullet[1]$, where $B_\bullet=(B,\ad_\mubar)$. Moreover, \labelcref{equation-interesting} shows the mapping cone of $-\delbar$ is isomorphic to $B_\bullet[2]$ by $\phi$. Then the inclusion of $B_\bullet[1]$ into the mapping cone of $-\delbar$ is identified with $\delbar: B_\bullet[1]\to B_\bullet[2]$, and the projection from the mapping cone of $-\delbar$ onto $B_\bullet[1]$ is identified with $\delta: B_\bullet[2]\to B_\bullet[1]$ which takes $\del x+\delbar y$ to $x$. It follows we have an exact triangle
\[
B_\bullet\xrightarrow{-\delbar} B_\bullet[1]\xrightarrow{\delbar} B_\bullet[2]\xrightarrow{\delta} B_\bullet[1].
\]
This exact triangle induces a long exact sequence in cohomology
\[
\cdots\to H^{k-2}(B_\bullet)\xrightarrow{-\delbar} H^{k-1}(B_\bullet)\xrightarrow{\delbar} H^k(B_\bullet)\xrightarrow{\delta} H^{k-1}(B_\bullet)\xrightarrow{-\delbar} H^k(B_\bullet)\to\cdots
\]

Now we can inductively prove $H(B_\bullet)=\Lambda(\delbar,[\del,\delbar])$. We note $\Lambda(\delbar,[\del,\delbar])$ is one-dimensional in each degree and spanned by powers of $[\del,\delbar]$ and $\delbar$ times powers of $[\del,\delbar]$. Assume the desired equality is proved in degrees $<k$. Observe that $\delbar$ vanishes on $\left(\Lambda(\delbar,[\del,\delbar])\right)^{\text{odd}}$, so if $k$ is even then from the above long exact sequence we have $\delta: H^k(B_\bullet)\to H^{k-1}(B_\bullet)$ is an isomorphism. On the other hand, if $k$ is odd, then $\delbar: H^{k-1}(B_\bullet)\to H^k(B_\bullet)$ is monic since $\delbar$ takes $\left(\Lambda(\delbar,[\del,\delbar])\right)^{\text{even}}$ injectively into $\left(\Lambda(\delbar,[\del,\delbar])\right)^{\text{odd}}\subset H^{\text{odd}}(B_\bullet)$ and thus the above long exact sequence implies $\delbar: H^{k-1}(B_\bullet)\to H^k(B_\bullet)$ is an isomorphism. So in either case we have $H^k(B_\bullet)\cong H^{k-1}(B_\bullet)$, and in particular $H^k(B_\bullet)$ is one-dimensional. But $H^k(B_\bullet)\supset\left(\Lambda(\delbar,[\del,\delbar])\right)^{k}$, therefore $H^k(B_\bullet)$ must be equal to $\left(\Lambda(\delbar,[\del,\delbar])\right)^{k}$. This finishes the inductive step and thus completes the proof.
\end{proof}

Now \Cref{quasiiso} follows as a corollary.

\begin{proof}[Proof of \Cref{quasiiso}]
One can explicitly find
\begin{align*}
    & H^1(\mathfrak{g},\ad_d)=\text{span}\{d,3\mubar+\delbar-\del-3\mu\},\\ 
   & H^1(\mathfrak{g}_{hol},\ad_d)=\text{span}\{d,\delbar-\del\},
\end{align*}
and $H^2(\mathfrak{g},\ad_d)=H^2(\mathfrak{g}_{hol},\ad_d)=0$. So $f$ induces isomorphisms on $H^1$ and $H^2$. It remains to show $f$ also induces an isomorphism on $H^k$ for $k>2$. But since $\mathfrak{g}_{hol}$ is concentrated in degrees $\le 2$, so it suffices to prove $H^k(\mathfrak{g},\ad_d)=0$ for $k>2$. For this, we use the previous proposition.

 Note that by \Cref{cor1} $\mathfrak{g}=\mathfrak{h}$ in degrees $\ge 2$, so we have $H^k(\mathfrak{g},\ad_\mubar)=H^k(\mathfrak{h},\ad_\mubar)=0$ for $k>2$. We claim the generalized Fr\"olicher spectral sequence of Cirici and Wilson \cite{Cirici-Wilson} implies $H^k(\mathfrak{g},\ad_d)=0$ for $k>2$. Indeed, their spectral sequence, even though designed for studying almost complex manifolds, applies more generally to (non-negatively) bigraded $\mathfrak{g}$-modules, and in particular applies to the adjoint action of $\mathfrak{g}$ on itself. Thus we have a spectral sequence
\[
E_1=H(H(\mathfrak{g},\ad_\mubar),\ad_\delbar)\Longrightarrow H(\mathfrak{g},\ad_d).
\]
Therefore the claim follows, and the proof is complete.
\end{proof}

Above, we fixed a differential on $\mathfrak{g}$ but in fact, there are many differentials on $\mathfrak{g}$ such as the inner differentials; i.e. $\ad_a$ for $a \in \mathfrak{g}_1$ such that $\ad_a^2=0$. Moreover, since $\mathfrak{g}_{hol}$ is abelian, all of its inner differentials are trivial. This means, $f:\mathfrak{g} \to \mathfrak{g}_{hol}$ is \textit{always} a morphism of differential graded Lie algebras no matter what \text{inner} differential we put on $\mathfrak{g}$. So it is natural to ask whether $f:(\mathfrak{g},\ad_a)\to (\mathfrak{g}_{hol},0)$ is a quasi-isomorphism for a general inner differential $\ad_a$.

To answer this question, we must first of all find all the inner differentials of $\mathfrak{g}$. Such a problem naturally lies in the context of deformation theory of graded Lie algebras as follows. Consider $\mathfrak{g}$ as a differential graded Lie algebra with the trivial differential $D \equiv 0$. Then the Maurer-Cartan equation for $(\mathfrak{g},0)$ exactly reads:
\[
[a,a]=0, a\in \mathfrak{g}_1.
\]
So $a\in \mathfrak{g}_1$ is a Maurer-Cartan solution for $(\mathfrak{g},0)$ if and only if $\ad_a$ is a inner differential on $\mathfrak{g}$.

Denote by $MC(\mathfrak{g})$ the algebraic set of Maurer-Cartan solutions for $(\mathfrak{g},0)$. Notice that if $a\in MC(\mathfrak{g})$ then so is $\lambda a$ for all $\lambda\in \mathbb{C}$. This mean $MC(\mathfrak{g})$ is the affine cone over its projectification $\mathbb{P}MC(\mathfrak{g})$.

\begin{proposition} \label{twcubic} The map $\mathbb{P}^1\to \mathbb{P}MC(\mathfrak{g})$, $[s {:} t]\mapsto d_{s,t}=s^3\mubar+s^2 t\delbar+s t^2\del+t^3\mu$ is an isomorphism. In particular $\mathbb{P}MC(\mathfrak{g})$ is a twisted cubic.
\end{proposition}

\begin{proof}
We express $a$ in the basis: $a=x\mubar + y\delbar + z\del + w \mu$. Then, by expanding the (quadratic) equation $[a,a]=0$, we have that $a$ is a solution to this Maurer-Cartan equation if and only if the following three quadratic equations are satisfied:
\begin{align*}
    &xz-y^2 = 0, \\
    &yw-z^2 = 0, \\
    &xw-yz = 0.
\end{align*}

These are exactly the equations in $\mathbb{P}^3$ defining the standard twisted cubic. The map $[s {:} t]\mapsto d_{s,t}=s^3\mubar+s^2 t\delbar+s t^2\del+t^3\mu$ is a parametrization of the twisted cubic.
\end{proof}

\begin{remark}
Compare the above to \cite[Lemma 3.1]{TT} where a different parametrization of $MC(\mathfrak{g})$ is obtained. Also, under the parametrization given in \Cref{twcubic}, we have $d_{1,0}=\mubar,d_{0,1} = \mu$ and $d_{1,1}=d$.
\end{remark}

It follows from the above proposition that $MC(\mathfrak{g})$ is in bijection with $\mathbb{A}^2$. This is in fact expected from deformation theory point of view. To elaborate, we first observe that the deformation problems associated to inner differentials are all \textit{equivalent}. Indeed, if $[\delta,-]$ is an inner differential, its corresponding Maurer-Cartan equation is
\[
[\delta+a,\delta+a]=0, a\in\mathfrak{g}_1.
\]
Thus $a\mapsto \delta+a$ establishes an isomorphism $MC(\mathfrak{g},\ad_\delta)\cong MC(\mathfrak{g})$. In particular we have $MC(\mathfrak{g})\cong MC(\mathfrak{g},\ad_d)$. Recall \Cref{quasiiso} shows $(\mathfrak{g},\ad_d)$ is quasi-isomorphic to $(\mathfrak{g}_{hol},0)$. Therefore by Goldman-Millson's theorem \cite{Goldman-Millson}, $MC(\mathfrak{g},\ad_d)$ is in bijection with $MC(\mathfrak{g}_{hol})$ (the gauge actions are trivial since $\mathfrak{g}_0=(\mathfrak{g}_{hol})_0=0$). The latter is easily seen to be isomorphic to $\mathbb{A}^2$.

We note that, even though $MC(\mathfrak{g})$ is in bijection with $\mathbb{A}^2$, it is \textit{not} isomorphic to $\mathbb{A}^2$ as algebraic varieties, since $MC(\mathfrak{g})$ is singular at its cone point while $\mathbb{A}^2$ is smooth.

Now that we have found all the inner differentials $d_{s,t}$ on $\mathfrak{g}$, we would like to determine $H(\mathfrak{g},\ad_{d_{s,t}})$. For simplicity of notation, henceforth we shall assume it is understood that the action of $d_{s,t}$ on $\mathfrak{g}$ is through the adjoint action, and so we simply write $d_{s,t}$ for $\ad_{d_{s,t}}$.

The first cohomology $H^1(\mathfrak{g},d_{s,t})$ coincides with the group of $1$-cocycles $Z^1(\mathfrak{g},d_{s,t})$ since $\mathfrak{g}_0=0$. By \cite{Goldman-Millson}, $Z^1(\mathfrak{g},d_{s,t})$ is the Zariski tangent space of $MC(\mathfrak{g})$ at $d_{s,t}$. So for $(s,t)\neq (0,0)$, $H^1(\mathfrak{g},d_{s,t})$ is $2$-dimensional and spanned by
\begin{align*}
    \frac{\del}{\del s}d_{s,t}&=3s^2\mubar+2st\delbar+t^2\del,\\
    \frac{\del}{\del t}d_{s,t}&=3t^2\mu+2ts\del+s^2\delbar.
\end{align*}
Alternatively notice that $[d_{s,t},d_{s,t}]=0$ implies $0=\frac{\del}{\del s}[d_{s,t},d_{s,t}]=2[\frac{\del}{\del s}d_{s,t},d_{s,t}]$ and similarly $2[\frac{\del}{\del t}d_{s,t},d_{s,t}]=0$. Hence $\frac{\del}{\del s}d_{s,t}$ and $\frac{\del}{\del t}d_{s,t}$ are contained in $H^1(\mathfrak{g},d_{s,t})$. It is an easy exercise for the reader to verify they are linearly independent and span $H^1(\mathfrak{g},d_{s,t})$.

The following lemma picks out a geometrically interesting basis for $H^1(\mathfrak{g},d_{s,t})$.

\begin{lemma}\label{goodbasis}
For $st\neq 0$, $d_{s,t}$ and $$d_{s,t}^J=\sqrt{-1}(3s^3\mubar+s^2 t\delbar-st^2\del-3t^3\mu)$$ span $H^1(\mathfrak{g},d_{s,t})$. Moreover
\begin{enumerate}[(i)]
    \item $d_{s,t}$ is a real operator (i.e. it is conjugate to itself) if and only if $d_{s,t}^J$ is a real operator;
    \item $[d_{s,t}^J,d_{s,t}^J]=0$ if and only if $[\del,\del]=[\delbar,\delbar]=[\del,\delbar]=0$.
\end{enumerate}
\end{lemma}
\begin{proof}
Note that
\begin{align*}
    3d_{s,t}&=s \frac{\del}{\del s}d_{s,t}+t\frac{\del}{\del s}d_{s,t},\\
    \sqrt{-1}d_{s,t}^J&=-s \frac{\del}{\del s}d_{s,t}+t\frac{\del}{\del s}d_{s,t}.\end{align*}
The rest is straightforward.
\end{proof}
\begin{remark}
In particular $H^1(\mathfrak{g},d)$ is spanned $d$ and $d^J = \sqrt{-1}(3 \mubar + \delbar - \del - 3 \mu)$. The authors learned from Scott Wilson why this is geometrically significant: the action of the operator $d^J$ on differential forms of an almost complex manifold $(M,J)$ coincides with the (graded) commutator $\mathcal{L}_J=[d,J]$, where $J$ acts on differential forms by extending its action on $1$-forms to an derivation on the de Rham algebra of all forms.
See \cite{CKT19} for more on the operator $\mathcal{L}_J$.
\end{remark}

Recall we have proved for $(s,t)=(1,0)$ and $(1,1)$ that $H^k(\mathfrak{g},d_{s,t})$ vanishes for $k>2$. Then by symmetry and $d_{s,0}=s^3d_{1,0}$, $d_{0,t}=t^3 d_{0,1}$, the same holds for all $(s,0),(0,t)$ with $s,t\neq 0$. It turns out the same is again true for all $(s,t)\neq (0,0)$. This follows immediately from the lemma below by considering the adjoint action of $\mathfrak{g}$ on itself.

\begin{lemma}
Let $V$ be a bigraded module of $\mathfrak{g}$, that is $V$ carries a bigrading and the $\mathfrak{g}$-action on $V$ is a bigraded one. Then the cohomology of $V$ with respect to $d_{s,t}$, $H(V,d_{s,t})$, is naturally isomorphic to $H(V,d)$ for $st\neq 0$.
\end{lemma}
\begin{proof}
Let $\varphi_{s,t}: V^{p,q} \to V^{p,q}$ be multiplication by $s^{p+2q} t^{2p+q}$. It is straightforward to verify that $d_{s,t}$ is conjugate to $d$ by $\varphi_{s,t}$; i.e. $\varphi_{s,t} d = d_{s,t} \varphi_{s,t}$. Thus, $\varphi_{s,t}$ induces the desired isomorphism on cohomology.
\end{proof}

\begin{corollary}
$f:(\mathfrak{g},d_{s,t})\to (\mathfrak{g}_{hol},0)$ is a quasi-isomorphism provided $st\neq 0$.
\end{corollary}
\begin{proof}
Using the basis given by \Cref{goodbasis}, $f$ is seen to be an isomorphism on $H^1$. Since higher cohomology groups vanish on both sides, this proves the statement.
\end{proof}
\begin{remark}
For $st=0$, $f$ is no longer a quasi-isomorphism. For example, $H^1(\mathfrak{g},\mubar)$ is spanned by $\mubar$ and $\delbar$ but $f$ takes $\mubar$ to zero. The behavior of the quotient map $f$ naturally puts a Whitney stratification on $MC(\mathfrak{g})$ which in coordinate $(s,t)$ is $\{st\neq 0\}\bigsqcup \{st=0\text{ but }(s,t)\neq (0,0)\}\bigsqcup\{(0,0)\}$. The nullity of the induced map of $f$ on $H^1$ is constant on each stratum.
\end{remark}

\section{Concluding Remarks}\label{sec4}

As mentioned in the introduction, understanding the structure and cohomology of these algebras generated by $\mubar,\delbar,\del,\mu$ is the first step in a larger program of understanding almost complex geometry. The next step is to consider the representations of these algebras since we naturally want to consider their actions on the differential forms of almost complex manifolds. The representations for $\mathfrak{g}_{hol}$ are relatively simple \cite{Stelzig}: they can be decomposed into ``dots'', ``squares'' and ``zigzags''. In contrast, the following example shows the representation theory for $\mathfrak{g}$ can be complicated. Let $\alpha, \beta, \gamma \in \mathbb{C}$.

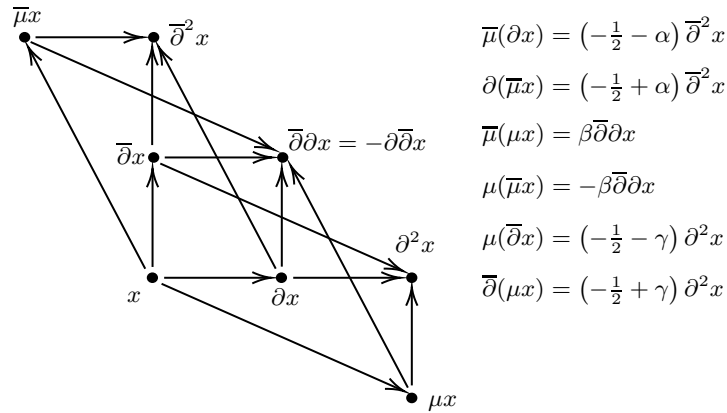
\begin{center}
\tikzset{every picture/.style={line width=0.75pt}} 

\begin{tikzpicture}[x=0.65pt,y=0.65pt,yscale=-1,xscale=1]

\draw  [fill={rgb, 255:red, 0; green, 0; blue, 0 }  ,fill opacity=1 ] (239.24,248.41) .. controls (239.24,247.03) and (240.42,245.91) .. (241.88,245.91) .. controls (243.35,245.9) and (244.54,247.01) .. (244.55,248.38) .. controls (244.56,249.76) and (243.38,250.88) .. (241.91,250.89) .. controls (240.45,250.9) and (239.25,249.79) .. (239.24,248.41) -- cycle ;
\draw    (247.18,251.92) -- (385.27,313.56) ;
\draw [shift={(387.09,314.38)}, rotate = 204.06] [color={rgb, 255:red, 0; green, 0; blue, 0 }  ][line width=0.75]    (10.93,-3.29) .. controls (6.95,-1.4) and (3.31,-0.3) .. (0,0) .. controls (3.31,0.3) and (6.95,1.4) .. (10.93,3.29)   ;
\draw  [fill={rgb, 255:red, 0; green, 0; blue, 0 }  ,fill opacity=1 ] (388.57,318.2) .. controls (388.56,316.83) and (389.74,315.7) .. (391.21,315.7) .. controls (392.67,315.69) and (393.87,316.8) .. (393.87,318.18) .. controls (393.88,319.55) and (392.7,320.67) .. (391.24,320.68) .. controls (389.77,320.69) and (388.58,319.58) .. (388.57,318.2) -- cycle ;
\draw  [fill={rgb, 255:red, 0; green, 0; blue, 0 }  ,fill opacity=1 ] (165.7,108.83) .. controls (165.69,107.45) and (166.87,106.33) .. (168.33,106.32) .. controls (169.8,106.32) and (170.99,107.43) .. (171,108.8) .. controls (171.01,110.18) and (169.83,111.3) .. (168.36,111.31) .. controls (166.9,111.31) and (165.7,110.2) .. (165.7,108.83) -- cycle ;
\draw  [fill={rgb, 255:red, 0; green, 0; blue, 0 }  ,fill opacity=1 ] (313.53,248.41) .. controls (313.53,247.03) and (314.71,245.91) .. (316.17,245.91) .. controls (317.64,245.9) and (318.83,247.01) .. (318.84,248.38) .. controls (318.85,249.76) and (317.67,250.88) .. (316.2,250.89) .. controls (314.74,250.9) and (313.54,249.79) .. (313.53,248.41) -- cycle ;
\draw  [fill={rgb, 255:red, 0; green, 0; blue, 0 }  ,fill opacity=1 ] (388.57,248.41) .. controls (388.56,247.03) and (389.74,245.91) .. (391.21,245.91) .. controls (392.67,245.9) and (393.87,247.01) .. (393.87,248.38) .. controls (393.88,249.76) and (392.7,250.88) .. (391.24,250.89) .. controls (389.77,250.9) and (388.58,249.79) .. (388.57,248.41) -- cycle ;
\draw  [fill={rgb, 255:red, 0; green, 0; blue, 0 }  ,fill opacity=1 ] (239.99,178.62) .. controls (239.98,177.24) and (241.16,176.12) .. (242.63,176.11) .. controls (244.09,176.11) and (245.28,177.22) .. (245.29,178.59) .. controls (245.3,179.97) and (244.12,181.09) .. (242.65,181.1) .. controls (241.19,181.11) and (239.99,180) .. (239.99,178.62) -- cycle ;
\draw  [fill={rgb, 255:red, 0; green, 0; blue, 0 }  ,fill opacity=1 ] (239.99,108.83) .. controls (239.98,107.45) and (241.16,106.33) .. (242.63,106.32) .. controls (244.09,106.32) and (245.28,107.43) .. (245.29,108.8) .. controls (245.3,110.18) and (244.12,111.3) .. (242.65,111.31) .. controls (241.19,111.31) and (239.99,110.2) .. (239.99,108.83) -- cycle ;
\draw  [fill={rgb, 255:red, 0; green, 0; blue, 0 }  ,fill opacity=1 ] (314.28,178.62) .. controls (314.27,177.24) and (315.45,176.12) .. (316.92,176.11) .. controls (318.38,176.11) and (319.58,177.22) .. (319.58,178.59) .. controls (319.59,179.97) and (318.41,181.09) .. (316.94,181.1) .. controls (315.48,181.11) and (314.29,180) .. (314.28,178.62) -- cycle ;
\draw    (246.43,182.13) -- (384.52,243.77) ;
\draw [shift={(386.35,244.59)}, rotate = 204.06] [color={rgb, 255:red, 0; green, 0; blue, 0 }  ][line width=0.75]    (10.93,-3.29) .. controls (6.95,-1.4) and (3.31,-0.3) .. (0,0) .. controls (3.31,0.3) and (6.95,1.4) .. (10.93,3.29)   ;
\draw    (171.8,111.56) -- (309.89,173.2) ;
\draw [shift={(311.72,174.02)}, rotate = 204.06] [color={rgb, 255:red, 0; green, 0; blue, 0 }  ][line width=0.75]    (10.93,-3.29) .. controls (6.95,-1.4) and (3.31,-0.3) .. (0,0) .. controls (3.31,0.3) and (6.95,1.4) .. (10.93,3.29)   ;
\draw    (248.66,248.43) -- (308.33,248.43) ;
\draw [shift={(310.33,248.43)}, rotate = 180] [color={rgb, 255:red, 0; green, 0; blue, 0 }  ][line width=0.75]    (10.93,-3.29) .. controls (6.95,-1.4) and (3.31,-0.3) .. (0,0) .. controls (3.31,0.3) and (6.95,1.4) .. (10.93,3.29)   ;
\draw    (322.96,248.43) -- (382.62,248.43) ;
\draw [shift={(384.62,248.43)}, rotate = 180] [color={rgb, 255:red, 0; green, 0; blue, 0 }  ][line width=0.75]    (10.93,-3.29) .. controls (6.95,-1.4) and (3.31,-0.3) .. (0,0) .. controls (3.31,0.3) and (6.95,1.4) .. (10.93,3.29)   ;
\draw    (249.01,178.59) -- (308.67,178.59) ;
\draw [shift={(310.67,178.59)}, rotate = 180] [color={rgb, 255:red, 0; green, 0; blue, 0 }  ][line width=0.75]    (10.93,-3.29) .. controls (6.95,-1.4) and (3.31,-0.3) .. (0,0) .. controls (3.31,0.3) and (6.95,1.4) .. (10.93,3.29)   ;
\draw    (241.98,242.15) -- (241.98,186.92) ;
\draw [shift={(241.98,184.92)}, rotate = 90] [color={rgb, 255:red, 0; green, 0; blue, 0 }  ][line width=0.75]    (10.93,-3.29) .. controls (6.95,-1.4) and (3.31,-0.3) .. (0,0) .. controls (3.31,0.3) and (6.95,1.4) .. (10.93,3.29)   ;
\draw    (316.39,241.97) -- (316.39,186.74) ;
\draw [shift={(316.39,184.74)}, rotate = 90] [color={rgb, 255:red, 0; green, 0; blue, 0 }  ][line width=0.75]    (10.93,-3.29) .. controls (6.95,-1.4) and (3.31,-0.3) .. (0,0) .. controls (3.31,0.3) and (6.95,1.4) .. (10.93,3.29)   ;
\draw    (241.76,171.4) -- (241.76,116.17) ;
\draw [shift={(241.76,114.17)}, rotate = 90] [color={rgb, 255:red, 0; green, 0; blue, 0 }  ][line width=0.75]    (10.93,-3.29) .. controls (6.95,-1.4) and (3.31,-0.3) .. (0,0) .. controls (3.31,0.3) and (6.95,1.4) .. (10.93,3.29)   ;
\draw    (237.77,243.19) -- (171.09,115.85) ;
\draw [shift={(170.16,114.08)}, rotate = 62.36] [color={rgb, 255:red, 0; green, 0; blue, 0 }  ][line width=0.75]    (10.93,-3.29) .. controls (6.95,-1.4) and (3.31,-0.3) .. (0,0) .. controls (3.31,0.3) and (6.95,1.4) .. (10.93,3.29)   ;
\draw    (312.9,243.15) -- (246.22,115.81) ;
\draw [shift={(245.29,114.04)}, rotate = 62.36] [color={rgb, 255:red, 0; green, 0; blue, 0 }  ][line width=0.75]    (10.93,-3.29) .. controls (6.95,-1.4) and (3.31,-0.3) .. (0,0) .. controls (3.31,0.3) and (6.95,1.4) .. (10.93,3.29)   ;
\draw    (388.98,312.21) -- (322.3,184.86) ;
\draw [shift={(321.37,183.09)}, rotate = 62.36] [color={rgb, 255:red, 0; green, 0; blue, 0 }  ][line width=0.75]    (10.93,-3.29) .. controls (6.95,-1.4) and (3.31,-0.3) .. (0,0) .. controls (3.31,0.3) and (6.95,1.4) .. (10.93,3.29)   ;
\draw    (174.61,108.83) -- (234.27,108.83) ;
\draw [shift={(236.27,108.83)}, rotate = 180] [color={rgb, 255:red, 0; green, 0; blue, 0 }  ][line width=0.75]    (10.93,-3.29) .. controls (6.95,-1.4) and (3.31,-0.3) .. (0,0) .. controls (3.31,0.3) and (6.95,1.4) .. (10.93,3.29)   ;
\draw    (391.24,310.91) -- (391.24,255.68) ;
\draw [shift={(391.24,253.68)}, rotate = 90] [color={rgb, 255:red, 0; green, 0; blue, 0 }  ][line width=0.75]    (10.93,-3.29) .. controls (6.95,-1.4) and (3.31,-0.3) .. (0,0) .. controls (3.31,0.3) and (6.95,1.4) .. (10.93,3.29)   ;

\draw (225.58,255) node [anchor=north west][inner sep=0.75pt]  [font=\small]  {$x$};
\draw (160,89.38) node [anchor=north west][inner sep=0.75pt]  [font=\small]  {$\mubar x$};
\draw (250.14,97.15) node [anchor=north west][inner sep=0.75pt]  [font=\small]  {$\overline{\partial }^{2} x$};
\draw (219.71,168) node [anchor=north west][inner sep=0.75pt]  [font=\small]  {$\overline{\partial } x$};
\draw (308.63,252.69) node [anchor=north west][inner sep=0.75pt]  [font=\small]  {$\partial x$};
\draw (399,315) node [anchor=north west][inner sep=0.75pt]  [font=\small]  {$\mu x$};
\draw (380,220) node [anchor=north west][inner sep=0.75pt]  [font=\small]  {$\partial^{2} x$};
\draw (318.48,160) node [anchor=north west][inner sep=0.75pt]  [font=\small]  {$\overline{\partial} \partial x=-\partial\overline{\partial}x$};

\draw (430,95) node [anchor=north west][inner sep=0.75pt]  [font=\small]  {$\overline{\mu}( \partial x)  = \left( -\frac{1}{2} -\alpha \right)\overline{\partial}^{2} x$};
\draw (430,125) node [anchor=north west][inner sep=0.75pt]  [font=\small]  {$\partial (\overline{\mu } x) = \left( -\frac{1}{2} +\alpha \right)\overline{\partial}^{2} x$};
\draw (430,155) node [anchor=north west][inner sep=0.75pt]  [font=\small]  {$\overline{\mu }( \mu x) = \beta \overline{\partial}\partial x$};
\draw (430,185) node [anchor=north west][inner sep=0.75pt]  [font=\small]  {$\mu (\overline{\mu } x)  =-\beta \overline{\partial}\partial  x$};
\draw (430,215) node [anchor=north west][inner sep=0.75pt]  [font=\small]  {$\mu (\overline{\partial } x) =\left( -\frac{1}{2} -\gamma \right)  \partial^{2} x$};
\draw (430,245) node [anchor=north west][inner sep=0.75pt]  [font=\small]  {$\overline{\partial}(\mu x) =\left( -\frac{1}{2} +\gamma \right) \partial^{2} x$};
\end{tikzpicture}

\end{center}
Here each node represents a basis vector and the arrows correspond to the operators $\mubar,\delbar,\del,\mu$, which are drawn in the way indicating their bidegrees. Arrows corresponding to vanishing operators, such as $\mubar(\mubar x)=0$, are omitted from the diagram.

This family of representations by construction descends to a family of faithful representations of the 6-dimensional quotient $\mathfrak{g}/([\del,\delbar])$ of $\mathfrak{g}$. The Lie subalgebra of $\mathfrak{g}/([\del,\delbar])$ generated by $\del,\delbar$ is isomorphic to the homotopy Lie algebra of $\mathbb{C}\mathbb{P}^1\times\mathbb{C}\mathbb{P}^1$ tensored with $\mathbb{C}$. To understand representations of $\mathfrak{g}$, it should be necessary as a first step to analyze representations of this special quotient $\mathfrak{g}/([\del,\delbar])$. As simple as $\mathfrak{g}/([\del,\delbar])$ may appear, it is nilpotent and general representation theory for nilpotent graded Lie algebras is not well-developed. However, for this specific Lie algebra, classification of representations may be possible, at least in low dimensions.

\bibliographystyle{alpha}
\bibliography{ref}

\begin{thebibliography}{CKT19}

\bibitem[CKT19]{CKT19}
Ki~Fung Chan, Spiro Karigiannis, and Chi~Cheuk Tsang.
\newblock Cohomologies on almost complex manifolds and the $\del\delbar$-lemma.
\newblock {\em Asian Journal of Mathematics}, 23(4):561--584, 2019.

\bibitem[CW21]{Cirici-Wilson}
Joana Cirici and Scott~O. Wilson.
\newblock Dolbeault cohomology for almost complex manifolds.
\newblock {\em Adv. Math.}, 391:107970, 52, 2021.

\bibitem[FHT01]{FHT}
Yves F\'{e}lix, Stephen Halperin, and Jean-Claude Thomas.
\newblock {\em Rational homotopy theory}, volume 205 of {\em Graduate Texts in
  Mathematics}.
\newblock Springer-Verlag, New York, 2001.

\bibitem[GM88]{Goldman-Millson}
William~M. Goldman and John~J. Millson.
\newblock The deformation theory of representations of fundamental groups of
  compact {K}\"{a}hler manifolds.
\newblock {\em Inst. Hautes \'{E}tudes Sci. Publ. Math.}, 67:43--96, 1988.

\bibitem[PSU22]{PSU20}
Dan Popovici, Jonas Stelzig, and Luis Ugarte.
\newblock Higher-page {H}odge theory of compact complex manifolds.
\newblock {\em Annali della Scuola Normale Superiore di Pisa, Classe di
  Scienze}, 2022.

\bibitem[Qui69]{Quillen}
Daniel Quillen.
\newblock Rational homotopy theory.
\newblock {\em Ann. of Math. (2)}, 90:205--295, 1969.

\bibitem[Ste21]{Stelzig}
Jonas Stelzig.
\newblock On the structure of double complexes.
\newblock {\em Journal of the London Mathematical Society}, 104(2):956--988,
  2021.

\bibitem[Ste23]{Stelzig2}
Jonas Stelzig.
\newblock Pluripotential homotopy theory.
\newblock \url{arxiv.org/abs/2302.08919}, 2023.

\bibitem[TT20]{TT}
Nicoletta Tardini and Adriano Tomassini.
\newblock Differential operators on almost-{H}ermitian manifolds and harmonic
  forms.
\newblock {\em Complex Manifolds}, 7(1):106--128, 2020.

\end{thebibliography}
\end{document}